\documentclass[12pt]{article}
\usepackage{multind}

\usepackage{color}

\usepackage[all]{xy}
\usepackage{amscd} 
\usepackage{amsmath,amssymb}
\usepackage{amsthm}
\usepackage{graphicx}
\usepackage{enumerate}
\usepackage{epstopdf}
\usepackage{mathtools}
\usepackage{xr}
\numberwithin{equation}{section}
\numberwithin{figure}{section}
\newtheorem{theorem}{Theorem}[section]

\newtheorem{lemma}[theorem]{Lemma}
\newtheorem{proposition}[theorem]{Proposition}

\newtheorem{fact}[theorem]{Fact}
\newtheorem{example}[theorem]{Example}
\theoremstyle{definition}
\newtheorem{definition}[theorem]{Definition}
\theoremstyle{remark}
\newtheorem{remark}[theorem]{Remark}
\theoremstyle{conjecture}

\newtheorem{problem}[theorem]{Problem}
\DeclareMathAlphabet{\mathpzc}{OT1}{pzc}{m}{it}
\newcommand{\Atbb}[4]
{{\widetilde{\mathbb{A}}}_{{#1},{#2},{#3}}^{{#4}}}
\newcommand{\Attbb}[4]
{\tilde{\tilde{{\mathbb{A}}}}_{{#1},{#2},{#3}}^{{#4}}}

\newcommand{\dual}[2]{\widehat{#1}_{\operatorname{#2}}}
\newcommand{\Exterior}{\mathchoice{{\textstyle\bigwedge}}%
    {{\bigwedge}}%
    {{\textstyle\wedge}}%
    {{\scriptstyle\wedge}}} 
\newcommand{\invHom}[3]{\operatorname{Hom}_{#1}({#2},{#3})}

\begin{document}
\title
{Conformal symmetry breaking 
 on differential forms and some applications}
\author{Toshiyuki Kobayashi 
\\
 Graduate School of Mathematical Sciences
 and Kavli IPMU
\\
 The University of Tokyo}
\date{}
\maketitle%

\abstract
{Rapid progress has been made recently
 on symmetry breaking operators
 for real reductive groups.  
Based on Program A--C for branching problems
 (T.~Kobayashi [Progr.~Math.~2015]),
 we illustrate a scheme 
 of the classification
 of (local and nonlocal) symmetry breaking operators
 by an example of conformal representations 
 on differential forms
 on the model space
 $(X,Y)=(S^n, S^{n-1})$, 
 which generalizes the scalar case
 (Kobayashi--Speh [Mem.~Amer.~Math.~Soc. 2015])
 and the case of local operators
 (Kobayashi--Kubo--Pevzner [Lect.~Notes Math. 2016]).  
Some applications to automorphic form theory, 
 motivations from conformal geometry, 
 and the methods of proof are also discussed.  
}
\vskip 1pc
\par\noindent
{\bf{Key words and phrases}}:\enspace
{branching rule, 
 conformal geometry, 
 reductive group, 
 symmetry breaking}

\newcommand{\isawatyp}
{
\left(
\begin{matrix} i \\ \lambda,\delta \end{matrix}
\,\left|\,
\begin{matrix} j \\ \nu, \varepsilon \end{matrix}
\right)
\right.
}

\newcommand{\isawavar}[6]
{
\left(
\begin{matrix} {#1} \\ {#2},{#3} \end{matrix}
\,\left|\,
\begin{matrix} {#4} \\ {#5}, {#6} \end{matrix}
\right)
\right.
}

\section{Branching problems---Stages A to C}
\label{sec:1}
Suppose $\Pi$ is an irreducible representation
 of a group $G$.  
We may regard $\Pi$ as a representation of its subgroup $G'$
 by restriction,
 which we denote by $\Pi|_{G'}$.  
The restriction $\Pi|_{G'}$ is not irreducible in general.  
In case it can be given as the direct sum
 of irreducible $G'$-modules, 
 the decomposition is called the {\it{branching law}}
 of the restriction $\Pi|_{G'}$.  

\begin{example}
[fusion rule]
Let $\pi_1$ and $\pi_2$ be representations of a group $H$.  
The outer tensor product $\Pi:=\pi_1 \boxtimes \pi_2$
 is a representation of the product group $G:= H \times H$, 
 and its restriction $\Pi|_{G'}$
 to the subgroup
 $G':={\operatorname{diag}}(H)$ is nothing
 but the tensor product representation
 $\pi_1 \otimes \pi_2$.  
In this case, 
 the branching law is called the {\it{fusion rule}}.  
\end{example}

For real reductive Lie groups
 such as $G=GL(n,{\mathbb{R}})$ or $O(p,q)$, 
 irreducible representations $\Pi$ are usually infinite-dimensional 
 and do not always possess highest weight vectors, 
 consequently,
 the restriction $\Pi|_{G'}$ to subgroups $G'$
may involve various (sometimes \lq\lq{wild}\rq\rq) aspects:
\begin{example}
\label{ex:fusion}
The fusion rule of two irreducible unitary principal series representations
 of $GL(n,{\mathbb{R}})$ $(n \ge 3)$ involve continuous spectrum
 and infinite multiplicities in the direct integral
 of irreducible unitary representations.  
\end{example}

By the {\it{branching problem}} 
 (in a wider sense than the usual), 
 we mean the problem of understanding 
 how the restriction $\Pi|_{G'}$ behaves 
 as a representation of the subgroup $G'$.  
We treat non-unitary representations $\Pi$ as well.  
In this case,  
 instead of considering the irreducible decomposition of the restriction 
 $\Pi|_{G'}$, 
we may investigate continuous $G'$-homomorphisms
\[
   T \colon \Pi|_{G'} \to \pi
\]
to irreducible representations $\pi$ of the subgroup $G'$.  
We call $T$ a {\it{symmetry breaking operator}}
 (SBO, for short).  
The dimension of the space of symmetry breaking operators
\[
   m(\Pi,\pi):= \dim_{\mathbb{C}}\invHom {G'} {\Pi|_{G'}}\pi
\]
 may be thought of
 as a variant of the \lq\lq{multiplicity}\rq\rq.  
Finding a formula of $m(\Pi,\pi)$ is a substitute
 of the branching law $\Pi|_{G'}$
 when $\Pi$ is not a unitary representation.

The author proposed in \cite{xkvogan} a program
 for branching problems
 in the following three stages:
\vskip 0.8pc
\par\noindent
{\bf{Stage A.}}\enspace
Abstract feature of the restriction $\Pi|_{G'}$.  
\par\noindent
{\bf{Stage B.}}\enspace
Branching laws.  
\par\noindent
{\bf{Stage C.}}\enspace
Construction of symmetry breaking operators.  
\vskip 0.8pc

Loosely speaking,
 Stage B concerns a decomposition
 of representations,
 whereas Stage C asks for a decomposition of vectors.

For \lq\lq{abstract features}\rq\rq\ of the restriction in Stage A,
 we may think of the following aspects:

\vskip 0.5pc
{\bf{A.1.}}\enspace
Spectrum of the restriction $\Pi|_{G'}$:
\begin{enumerate}
\item[$\bullet$]
(discretely decomposable case, \cite{xkInvent94, xkAnn98, xkInvent98})\enspace
branching problems could be studied purely
 algebraic and combinatorial approaches;
\item[$\bullet$]
(continuous spectrum)\enspace
branching problems may be of analytic feature
({\it{e.g.}}, Example \ref{ex:fusion}).  
\end{enumerate}
\vskip 0.5pc
{\bf{A.2.}}\enspace
Estimate of multiplicities for the restriction $\Pi|_{G'}$:
\begin{enumerate}
\item[$\bullet$]
multiplicities may be infinite 
(see Example \ref{ex:fusion});
\item[$\bullet$]
multiplicities may be at most one in special settings
 ({\it{e.g.}}, theta correspondence \cite{xhowe}, 
 Gross--Prasad conjecture \cite{GP}, 
 real forms of strong Gelfand pairs \cite{SunZhu}, 
 visible actions \cite{xkProgrMath13}, 
 {\it{etc.}}).  
\end{enumerate}
The goal of Stage A in branching problems 
 is to analyze aspects
 such as A.1 and A.2
 in complete generality.  
If multiplicities of the restriction $\Pi|_{G'}$
 are known {\it{a priori}}
 to be bounded in Stage A, 
 one might be tempted to find irreducible decompositions
 (Stage B), 
 and moreover to construct explicit symmetry breaking operators (Stage C).  
Thus, 
 results in Stage A might also serve as
 a foundation for further detailed study
 of the restriction $\Pi|_{G'}$
 (Stages B and C).

\vskip 0.8pc
This article is divided into three parts.  
First, 
 we discuss Stage A in Section \ref{sec:mult}
 with focus on multiplicities
 in both regular representations 
 on homogeneous spaces
 and branching problems based on a joint work
 \cite{xKOfm} with T.~Oshima, 
 and give some perspectives
 of the subject
 through the classification theory
 \cite{xKMt} joint with T.~Matsuki
 about the pairs $(G,G')$ 
 for which multiplicities in branching laws
 are always finite.

Second,
 we take $(G,G')$ to be $(O(n+1,1),O(n,1))$
 as an example of such pairs, 
 and explain the first test case
 for the classification problem of symmetry breaking operators
 (Stages B and C).  
The choice of our setting is motivated
 by conformal geometry,
 and is also related to the local Gross--Prasad
 conjecture \cite{GP, sbonGP}.  
We survey the classification theory of conformally covariant SBO
 for differential forms
 on the model space 
 $(X,Y)=(S^n,S^{n-1})$:
 for local operators
 based on a recent book \cite{KKP}
 with T.~Kubo and M.~Pevzner 
 in Section \ref{sec:DSBO}
 and for nonlocal operators 
 based on a recent monograph \cite{sbon} with B.~Speh
 and its generalization \cite{sbonvec}
 in Section \ref{sec:CSBO}.

In Section \ref{sec:period},
 we discuss an ongoing work with Speh on some applications
 of these results to a question from automorphic form theory, 
 in particular, 
 about the periods of irreducible representations
 with nonzero $({\mathfrak{g}},K)$-cohomologies.  
The resulting condition to admit periods
is compared with a recent $L^2$-theory
 \cite{BK} 
 joint with Y.~Benoist.

Detailed proofs of the new results
 in Sections \ref{sec:CSBO} and \ref{sec:period}
 will be given in separate papers \cite{xkresidue, sbonvec}.  

\vskip 1pc
{\bf{Notation.}}\enspace
${\mathbb{N}}=\{0,1,2,\cdots\}$.  

\section{Preliminaries: smooth representations}
We would like to treat non-unitary representations as well
 for the study of branching problems.  
For this we recall 
 some standard concepts
 of continuous representations of Lie groups.

Suppose $\Pi$ is a continuous representation 
 of $G$ on a Banach space $V$.  
A vector $v \in V$ is said
 to be {\it{smooth}}
 if the map
 $G \to V$, 
 $g \mapsto \Pi(g)v$ is of $C^{\infty}$-class.  
Let $V^{\infty}$ denote
 the space
 of smooth vectors
 of the representation $(\Pi,V)$.  
Then $V^{\infty}$ is a $G$-invariant dense subspace
 of $V$, 
 and $V^{\infty}$ carries 
 a Fr{\'e}chet topology
 with a family of semi-norms
$\|v\|_{i_1\cdots i_k}:=\|d\Pi(X_{i_1}) \cdots d\Pi(X_{i_k})v\|$, 
 where $\{X_1, \dots, X_n\}$ is a basis
 of the Lie algebra ${\mathfrak {g}}_0$ of $G$.  
Thus we obtain a continuous Fr{\'e}chet representation 
 $(\Pi^{\infty}, V^{\infty})$
 of $G$.

Suppose now that $G$ is a real reductive linear Lie group, 
 $K$ a maximal compact subgroup of $G$,
 and ${\mathfrak {g}}$ the complexification
 of the Lie algebra ${\mathfrak {g}}_0$ of $G$.  
Let ${\mathcal{HC}}$ denote the category
 of Harish-Chandra modules
 whose objects and morphisms are $({\mathfrak {g}}, K)$-modules
 of finite length
 and $({\mathfrak {g}}, K)$-homomorphisms, 
 respectively.  
Let $\Pi$ be a continuous representation 
 of $G$ on a complete locally convex topological vector space $V$.  
Assume that the $G$-module $\Pi$ is of finite length.  
We say $\Pi$ is {\it{admissible}}
 if 
\[
   \dim_{\mathbb{C}} \operatorname{Hom}_K(\tau, \Pi|_K)< \infty
\]
 for all irreducible finite-dimensional representations $\tau$ of $K$.  
We denote by $V_K$
 the space of $K$-finite vectors.  
Then $V_K \subset V^{\infty}$
 and the Lie algebra ${\mathfrak {g}}$
 leaves $V_K$ invariant.  
The resulting $({\mathfrak {g}}, K)$-module
 on $V_K$
 is called the underlying $({\mathfrak {g}}, K)$-module
 of $\Pi$, 
 and will be denoted by $\Pi_K$.

For any admissible representation $\Pi$
 on a Banach space $V$, 
 the smooth representation $(\Pi^{\infty}, V^{\infty})$
 depends only on the underlying 
 $({\mathfrak {g}}, K)$-module.  
We say $(\Pi^{\infty}, V^{\infty})$
 is an {\it{admissible smooth representation}}.  
By the Casselman--Wallach globalization theory, 
 $(\Pi^{\infty}, V^{\infty})$ has moderate growth, 
 and there is a canonical equivalence of categories
 between the category ${\mathcal{HC}}$
 of Harish-Chandra modules
 and the category of admissible smooth representations 
 of $G$ (\cite[Chap.~11]{W}).  
In particular,
 the Fr{\'e}chet representation
 $\Pi^{\infty}$ is uniquely 
 determined by its underlying
 $({\mathfrak {g}}, K)$-module.  
We say $\Pi^{\infty}$
 is the {\it{smooth globalization}}
 of $\Pi_K \in {\mathcal{HC}}$.

For simplicity,
 by an {\it{irreducible smooth representation}}, 
 we shall mean an irreducible admissible smooth representation
 of $G$.  
We denote by $\dual {G}{smooth}$
 the set of equivalence classes
 of irreducible smooth representations 
 of $G$.  
Via the underlying $({\mathfrak {g}}, K)$-modules, 
 we may regard the unitary dual $\dual G{}$ 
as a subset of $\dual G {smooth}$.

\section{Multiplicities in symmetry breaking}
\label{sec:mult}
Let $G \supset G'$ be a pair of real reductive groups.  
For $\Pi \in \dual G {smooth}$ and $\pi \in \dual {G'} {smooth}$, 
 we denote by $\invHom {G'} {\Pi|_{G'}}{\pi}$ 
 the space of symmetry breaking operators,
 and define the {\it{multiplicity}}
 (for smooth representation) by 
\begin{equation}
\label{eqn:mult}
   m(\Pi,\pi):= \dim_{\mathbb{C}}\invHom {G'} {\Pi|_{G'}}{\pi}
   \in {\mathbb{N}} \cup \{\infty\}.  
\end{equation}
Note that $m(\Pi,\pi)$ is well-defined 
 without the unitarity assumption on $\Pi$ and $\pi$.

We established a geometric criterion for multiplicities
 to be finite
 (more strongly, to be bounded) as follows:
\begin{theorem}
[\cite{xKOfm},  see also \cite{Ksuron, xkProg2014}]
\label{thm:PPBB}
Let $G \supset G'$ be a pair of real reductive algebraic Lie groups.  
\begin{enumerate}
\item[{\rm{(1)}}]
The following two conditions
 on the pair $(G,G')$ are equivalent:
\begin{enumerate}
\item[{\rm{(FM)}}]
{\rm{(finite multiplicities)}}\enspace
$m(\Pi, \pi)<\infty$ for all $\Pi \in \widehat G_{\operatorname{smooth}}$
 and $\pi \in \widehat {G'}_{\operatorname{smooth}}$;
\item[{\rm{(PP)}}]
{\rm{(geometry)}}
$(G \times G')/{\operatorname{diag}}(G')$ is real spherical.  
\end{enumerate}
\item[{\rm{(2)}}]
The following two conditions on the pair $(G,G')$ are equivalent:
\begin{enumerate}
\item[{\rm{(BM)}}]
{\rm{(bounded multiplicities)}}\enspace
There exists $C>0$ such that 
\[
   m(\Pi, \pi) \le C
\quad
\text{for all $\Pi \in \widehat G_{\operatorname{smooth}}$
 and $\pi \in \widehat {G'}_{\operatorname{smooth}}$;}
\]
\item[{\rm{(BB)}}]
{\rm{(complex geometry)}}
$(G_{\mathbb{C}} \times G_{\mathbb{C}}')/{\operatorname{diag}}(G_{\mathbb{C}}')$ is spherical.  
\end{enumerate}
\end{enumerate}
\end{theorem}

Here we recall that a connected complex manifold $X_{\mathbb{C}}$ 
 with holomorphic action of a complex reductive group $G_{\mathbb{C}}$
 is called {\it{spherical}} 
 if a Borel subgroup of $G_{\mathbb{C}}$
 has an open orbit in $X_{\mathbb{C}}$.  
There has been an extensive study
 of spherical varieties in algebraic geometry
 and finite-dimensional representation theory.  
In constant, 
 concerning the real setting, 
 in search of a good framework for global analysis 
 on homogeneous spaces
 which are broader than the usual 
 ({\it{e.g.}}, reductive symmetric spaces),
 the author proposed:
\begin{definition}
[{\cite{Ksuron}}]
\label{def:realsp}
\rm{
Let $G$ be a real reductive Lie group.  
We say a connected smooth manifold $X$ with smooth $G$-action
 is {\it{real spherical}}
 if a minimal parabolic subgroup $P$
 of $G$ has an open orbit in $X$.  
}
\end{definition}
We discovered in \cite{Ksuron, xKOfm}
 that these geometric properties
 (spherical/real spherical) are exactly the conditions
 that a reductive group $G$ has a \lq\lq{strong grip}\rq\rq\
 of the space of functions on $X$
 in the context of multiplicities of (infinite-dimensional)
 irreducible representations
 occurring in the regular representation
 of $G$ on $C^{\infty}(X)$:
\begin{theorem}
[{\cite[Thms.~A and C]{xKOfm}}]
\label{fact:HP}
Suppose $G$ is a real reductive linear Lie group,
 $H$ is an algebraic reductive subgroup, 
 and $X=G/H$. 
\begin{enumerate}
\item[{\rm{(1)}}] 
The homogeneous space $X$ 
 is real spherical
 if and only if 
\[
\text{
$\dim_{\mathbb{C}} \invHom{G}{\pi}{C^{\infty}(X)}< \infty$
 for all $\pi \in \dual G {smooth}$.  
}
\]
\item[{\rm{(2)}}]
The complexification $X_{\mathbb{C}}$
 is spherical 
 if and only if 
\[
\sup_{\pi \in \dual G {smooth}}
 \dim_{\mathbb{C}} \invHom G{\pi}{C^{\infty}(X)}
< \infty.  
\]
\end{enumerate}
\end{theorem}

\par\noindent
{\bf{Methods of proof.}}\enspace
In \cite{xKOfm}, 
 we obtained not only the equivalences in Theorem \ref{fact:HP}
 but also quantitative estimates
 of the dimension.  
The proof for the upper estimate in \cite{xKOfm} uses the theory 
 of regular singularities
 of a system of partial differential equations
 by taking an appropriate compactification
 with normal crossing boundaries, 
 whereas the proof for the lower estimate uses 
 the construction of a \lq\lq{generalized Poisson transform}\rq\rq.  
Furthermore,
 these estimates hold 
 for the representations of $G$
 on the space of smooth sections 
 for equivariant vector bundles
 over $X=G/H$ without assuming 
 that $H$ is reductive.  
For instance,
 this applies also to the case
 where $H$ is a maximal unipotent subgroup of $G$, 
giving a Kostant--Lynch estimate the dimension of the space
 of Whittaker vectors
 (\cite[Ex.~1.4 (3)]{xKOfm}).

Back to Theorem \ref{thm:PPBB} on branching problems, 
 the geometric estimates of multiplicities
 is proved by applying Theorem \ref{fact:HP}
 to the pair $(G \times G', {\operatorname{diag}}(G'))$
 together with some careful arguments
 on topological vector spaces
 (\cite[Thm.~4.1]{xkProg2014}).  

\vskip 0.8pc
\par\noindent
{\bf{Classification theory.}}\enspace
Theorem \ref{thm:PPBB} serves Stage A 
 in branching problems,
 and singles out nice settings in which we could expect to 
 go further on Stages B and C
 of the detailed study of symmetry breaking.

So it would be useful
 to develop a classification theory of pairs $(G,G')$
 for which the geometric criteria (PP) or (BB) in Theorem \ref{thm:PPBB}
 are satisfied.  
\begin{enumerate}
\item[$\bullet$]
The geometric criterion (BB) in Theorem \ref{thm:PPBB}
 appeared in the context
 of {\it{finite-dimensional}} representations
 already in 1970s, 
 and such pairs $(G_{\mathbb{C}},G_{\mathbb{C}}')$
 were classified infinitesimally,
 see \cite{Kr1}.  
The classification of real forms $(G,G')$
 satisfying the condition (BB)
 follows readily from 
 that of complex pairs $(G_{\mathbb{C}},G_{\mathbb{C}}')$, 
 see \cite{xKMt}.  
Sun--Zhu \cite{SunZhu} proved
 that the constant $C$ in Theorem \ref{thm:PPBB} can be taken
 to be one
 ({\it{multiplicity-free theorem}})
 in many of real forms $(G,G')$,
 see \cite[Rem.~2.2]{sbonGP} for multiplicity-two results 
 for some other real forms.  
\item[$\bullet$]
The pairs $({}^{\backprime} G \times {}^{\backprime} G,
 {\operatorname{diag}}({}^{\backprime} G))$
 for real reductive groups ${}^{\backprime} G$ satisfying the geometric criterion (PP)
 in Theorem \ref{thm:PPBB} were classified in \cite{Ksuron}.  
\item[$\bullet$]
More generally,
 symmetric pairs $(G,G')$ 
 satisfying the geometric criterion (PP)
 in Theorem \ref{thm:PPBB}
 was classified by the author
 and Matsuki \cite{xKMt}.  
The methods are a linearization technique
 and invariants of quivers.  
\end{enumerate}
In turn,
 these classification results give an {\it{a priori}} estimate
 of multiplicities in branching problems
 by Theorem \ref{thm:PPBB}.  
\begin{example}
[finite multiplicities for the fusion rule, {\cite[Ex.~2.8.6]{Ksuron}}, 
 see also {\cite[Cor.~4.2]{xkProg2014}}]
Suppose $G$ is a simple Lie group.  
Then the following two conditions are equivalent:
\begin{enumerate}
\item[{\rm{(i)}}]
$\dim_{\mathbb{C}}\invHom {G}{\pi_1 \otimes \pi_2}{\pi_3}<\infty$
 for all $\pi_1$, $\pi_2$, $\pi_3 \in \dual G{smooth}$;
\item[{\rm{(ii)}}]
$G$ is either compact or locally isomorphic to $SO(n,1)$.  
\end{enumerate}
\end{example}

\begin{example}
\label{ex:Opqr}
Let $(G,G')=(O(p+r,q), O(r)\times O(p,q))$.  
\begin{enumerate}
\item[{\rm{(1)}}]
$m(\Pi,\pi)<\infty$
for all $\Pi \in \dual G{smooth}$
 and $\pi \in \dual{G'}{smooth}$.  
\item[{\rm{(2)}}]
$m(\Pi,\pi) \le 1$
for all $\Pi \in \dual G{smooth}$
 and $\pi \in \dual{G'}{smooth}$
 iff $p+q+r \le 4$ or $r=1$.  
\end{enumerate}
\end{example}

See \cite{xkProg2014} for the further classification theory 
 of symmetric pairs $(G,G')$
 that guarantee finite multiplicity properties
 for symmetry breaking.  

\section{Conformally covariant SBOs}
\label{sec:CCSBO}
This section discusses a question
 on symmetry breaking with respect to a pair of conformal manifolds 
 $X \supset Y$.

Let $(X,g)$ be a Riemannian manifold. Suppose that a Lie group $G$ acts conformally on $X$. This means that there exists a positive-valued function $\Omega\in C^\infty(G\times X)$ (\index{B}{conformal factor}\emph{conformal factor}) such that
\index{A}{11zOmega@$\Omega(h,x)$, conformal factor|textbf}
\begin{equation*}
L_h^*g_{h\cdot x}=\Omega(h,x)^2g_x\quad \text{for all $h \in G$, $x \in X$},
\end{equation*}
where 
we write $L_h\colon X\to X, x\mapsto h\cdot x$ for the action of $G$ on $X$. 
When $X$ is oriented, we define a locally constant function
\index{A}{or@$\mathpzc{or}$|textbf}
\[
\mathpzc{or}\colon G\times X \longrightarrow \{\pm1\}
\] 
by $\mathpzc{or}(h)(x)=1$ if $(L_h)_{*x}\colon T_xX \longrightarrow T_{L_hx}X$ 
is orientation-preserving, and $=-1$ if it is orientation-reversing.

Since both the conformal factor $\Omega$
 and the orientation map $\mathpzc{or}$ satisfy
 cocycle conditions,
 we can form a family of representations 
$\varpi^{(i)}_{\lambda,\delta}$ of $G$ with parameters 
$\lambda \in{\mathbb{C}}$ and $\delta\in{\mathbb{Z}}/2{\mathbb{Z}}$
 on the space 
\index{A}{E10@$\mathcal E^i(X)$}
$\mathcal E^i(X)$ of differential $i$-forms on $X$ $(0\leq i \leq \dim X)$ 
 defined by
\index{A}{1pi@$\varpi^{(i)}_{u,\delta}$, conformal representation on $i$-forms|textbf}
\begin{equation}
\label{eqn:varpi}
\varpi^{(i)}_{\lambda,\delta}(h)\alpha:=\mathpzc{or}(h)^\delta\Omega(h^{-1},\cdot)^{\lambda}
L_{h^{-1}}^*\alpha,\quad (h\in G).
\end{equation}
The representation $\varpi^{(i)}_{\lambda,\delta}$ of the conformal group $G$ on $\mathcal E^i(X)$
will be simply denoted by
\index{A}{E1i@$\mathcal E^i(X)_{\lambda,\delta}$, 
conformal representation on $i$-forms on $X$|textbf}
$\mathcal E^i(X)_{\lambda,\delta}$,
 and referred to as the 
\index{B}{conformal representation on $i$-forms|textbf}
{\it{conformal representation}}
 on differential $i$-forms.

Suppose that $Y$ is an orientable submanifold.  
Then $Y$ is endowed with a Riemannian structure $g\vert_{Y}$ by restriction, 
and we can define in a similar way a family of representations
 $\mathcal E^j(Y)_{\nu, \varepsilon}$
 ($\nu\in{\mathbb{C}},\varepsilon\in{\mathbb{Z}}/2{\mathbb{Z}},
 0\leq j\leq\dim Y$)
 of the conformal group of $(Y,g{\vert}_{ Y})$.

We consider the full group of conformal diffeomorphisms
 and its subgroup defined as 
\index{A}{C1X@$\mathrm{Conf}(X)$|textbf}
\index{A}{C1XY@$\mathrm{Conf}(X;Y)$|textbf}
\begin{align}
\mathrm{Conf}(X)&:=\{\text{conformal diffeomorphisms of $(X, g)$}\},
\notag
\\
\mathrm{Conf}(X;Y)&:=\{
\varphi \in \mathrm{Conf}(X): \varphi(Y)=Y\}.
\label{eqn:confXY}
\end{align} 
Then there is a natural group homomorphism
\begin{equation}
\label{eqn:confXYY}
  \mathrm{Conf}(X;Y) \to \mathrm{Conf}(Y), 
 \quad
  \varphi \mapsto \varphi|_Y.  
\end{equation}

\begin{definition}
A linear map $T \colon {\mathcal{E}}^i(X)_{\lambda,\delta}
 \to {\mathcal{E}}^j(Y)_{\nu,\varepsilon}$
 is a {\it{conformally covariant symmetry breaking operator}}
 (conformally covariant SBO, for short)
 if $T$ intertwines the actions
 of the group ${\operatorname{Conf}}(X;Y)$.  
\end{definition}
We shall write 
\begin{alignat}{2}
&H \isawatyp &&:= {\operatorname{Hom}}_{{\operatorname{Conf}}(X;Y)}
             ({\mathcal{E}}^i(X)_{\lambda,\delta}|_{{\operatorname{Conf}}(X;Y)}, 
              {\mathcal{E}}^j(Y)_{\nu,\varepsilon})
\label{eqn:H}
\\
&\qquad\quad\cup &&
\notag
\\
&D \isawatyp &&:= {\operatorname{Diff}}_{{\operatorname{Conf}}(X;Y)}
             ({\mathcal{E}}^i(X)_{\lambda,\delta}|_{{\operatorname{Conf}}(X;Y)}, 
              {\mathcal{E}}^j(Y)_{\nu,\varepsilon})
\label{eqn:D}
\end{alignat}
for the space of continuous conformally covariant
 SBOs 
and its subspace of differential SBOs, 
 namely,
 those operators $T$ satisfying the local property:
 ${\operatorname{Supp}}(T\alpha) \subset {\operatorname{Supp}}(\alpha)$ 
 for all $\alpha \in {\mathcal{E}}^i(X)_{\lambda,\delta}$.  
This support condition is a generalization
 of Peetre's characterization \cite{Pe}
 of differential operators in the $X=Y$ case
 (\cite[Def.~2.1]{KP1}, for instance).

\vskip 0.8pc
We address a general problem motivated by 
 conformal geometry:
\begin{problem}
[conformally covariant symmetry breaking operators]
\label{prob:conf}
Let $X \supset Y$ are orientable Riemannian manifolds.  
\begin{enumerate}
\item[{\rm{(1)}}]
Determine when $H \isawatyp \ne \{0\}$.  
\item[{\rm{(2)}}]
Determine when $D \isawatyp \ne \{0\}$.  
\item[{\rm{(3)}}]
Construct an explicit basis
 of $H \isawatyp$ and $D \isawatyp$.  
\end{enumerate}
\end{problem}
Problem \ref{prob:conf} (1) and (2)
 may be thought of as Stage B
 of branching problems in Section \ref{sec:1}, 
 while Problem \ref{prob:conf} (3) as Stage C.

In the case where $X=Y$ and $i = j = 0$,
a classical prototype of such operators is 
a second order differential operator called the \index{B}{Yamabe operator}
Yamabe operator
\begin{equation*}
\Delta+\frac{n-2}{4(n-1)}\kappa\in \mathrm{Diff}_{\mathrm{Conf}(X)}(\mathcal E^0(X)_{\frac n2-1,\delta},\mathcal E^0(X)_{\frac n2+1,\delta}),
\end{equation*}
where $n$ is the dimension of the manifold $X$, 
$\Delta$ is the Laplacian,
 and $\kappa$ is the
scalar curvature, 
 see \cite[Thm.~A]{KO1}, for instance.
Conformally covariant differential operators of higher order are also known:
the Paneitz operator (fourth order) \cite{P08}, 
or more generally, the so-called 
\index{B}{GJMS operator}
GJMS operators \cite{GJMS} are such operators.
Turning to operators acting on  differential forms, 
 we observe that the exterior derivative $d$,
 the codifferential $d^*$,
 and the Hodge $\ast$ operator
 are also examples of conformally covariant operators on differential forms,
 namely,
 $j=i+1$, $i-1$, and $n-i$, respectively, 
 with an appropriate choice
 of the parameter $(\lambda,\nu,\delta, \varepsilon)$.
As is well-known, 
 Maxwell's equations in four-dimension
 can be expressed in terms of conformally 
 covariant operators on differential forms.  

Let us consider the general case where $X\neq Y$. 
{}From the viewpoint of conformal geometry,
 we are interested in \lq\lq{natural operators}\rq\rq\ $T$
 that persist for all pairs of Riemannian manifolds
 $X \supset Y$
 of fixed dimension.  
We note that Problem \ref{prob:conf} is trivial 
 for 
individual pairs $X \supset Y$ such that $\mathrm{Conf}(X;Y) = \{e\}$,
because any linear operator becomes automatically an SBO.  
In contrast, the larger ${\operatorname{Conf}}(X;Y)$ is, 
 the more constraints on $T$ will be imposed.  
Thus we highlight the case of large conformal groups as the first step to attack Problem \ref{prob:conf}.

In general, 
 the conformal group cannot be so large.  
We recall from \cite[Thms.~6.1 and 6.2]{K95} the upper estimate of the dimension of the conformal group:
\begin{fact}
\label{fact:dimConf}
Let $X$ be an $n$-dimensional compact Riemannian manifold of dimension 
 $n \ge 3$.  
Then 
$
 \dim {\operatorname{Conf}}(X) \le \frac 1 2 (n+1)(n+2).  
$
The equality holds
 if and only if $({\operatorname{Conf}}(X),X)$ is locally isomorphic to $(O(n+1,1), S^n)$.  
\end{fact}

Concerning a pair $(X,Y)$ of Riemannian manifolds, 
 we obtain the following.
\begin{proposition}
\label{prop:manXY}
Let $X \supset Y$ be Riemannian manifolds of dimension $n$ and $m$, 
 respectively.  
Then $\dim {\operatorname{Conf}}(X;Y) \le \frac 1 2 (m+1)(m+2)$.  
The equality holds
 if $X=S^n$ and $Y$ is a totally geodesic submanifold
 which is isomorphic to $S^m$.  
\end{proposition}
\begin{proof}
The first inequality follows from Fact \ref{fact:dimConf}
 via the group homomorphism \eqref{eqn:confXYY}.  
If $(X,Y)=(S^n,S^m)$, 
 then 
${\operatorname{Conf}}(X)$ and ${\operatorname{Conf}}(X;Y)$
 are locally isomorphic to $O(n+1,1)$ and $O(m+1,1)$, 
 respectively,
 whence the second assertion. 
\end{proof}
{}From now on,
 we shall consider the pair
\begin{equation}
\label{eqn:XYsph}
  (X,Y)=(S^n, S^{n-1}), 
\end{equation}
 as a model case with largest symmetries, 
 where $Y=S^{n-1}$ is embedded as a totally geodesic submanifold of $X=S^n$.  
As mentioned, 
 the pair $({\operatorname{Conf}}(X), {\operatorname{Conf}}(X;Y))$ is 
 locally isomorphic to the pair 
\begin{equation}
\label{eqn:Lorentz}
  (G,G')=(O(n+1,1), O(n,1)).  
\end{equation}
We remind that this pair appeared in Section \ref{sec:mult}
 on branching problems, 
 see the case 
 where $r=1$
 in Example \ref{ex:Opqr}.  
As an {\it{a priori}} estimate in Stage A, 
 see Theorem \ref{thm:PPBB} (2), 
 Example \ref{ex:Opqr}, 
 \cite[Thm.~2.6]{KKP}, 
 and \cite{SunZhu}, 
 we have
\begin{equation}
\label{eqn:multfour}
 \dim_{\mathbb{C}} H\isawatyp
  \le 4 
\quad
\text{for any $(i,j,\lambda,\nu,\delta,\varepsilon)$.  }
\end{equation}
In turn,
 the estimate \eqref{eqn:multfour} gives an upper bound
 for the dimension of the space of \lq\lq{natural}\rq\rq\ conformal 
 covariant SBOs, 
 ${\mathcal{E}}^i(X)_{\lambda,\delta} \to {\mathcal{E}}^j(Y)_{\nu,\varepsilon}$
 that persist for all pairs $X \supset Y$
 of codimension one.  
In the next two sections,
 we explain briefly a solution to Problem \ref{prob:conf}
 (Stages B and C)
 in the model case \eqref{eqn:XYsph}.

\section{Classification theory of conformally covariant differential SBOs}
\label{sec:DSBO}
In the case where symmetry breaking operators
 are given as differential operators,
 Problem \ref{prob:conf} 
 in the model space \eqref{eqn:XYsph}
 was solved 
 in a joint work \cite{KKP}
 with Kubo and Pevzner.  
In this section,
 we introduce its flavors briefly.  
First of all, 
 the solution to Problem \ref{prob:conf} (2), 
 a question in Stage B
 of branching problems, 
 may be stated as follows.  

\begin{theorem}
\label{thm:Dnonzero}
Suppose $n \ge 3$, 
 $0 \le i \le n$, 
 $0 \le j \le n-1$,  
 $\lambda, \nu \in {\mathbb{C}}$, 
 and $\delta, \varepsilon \in \{\pm\}$.  
Then the following three conditions on 6-tuple
 $(i,j,\lambda,\nu,\delta,\varepsilon)$ are equivalent:
\begin{enumerate}
\item[{\rm{(i)}}]
$D \isawatyp \ne \{0\}$.  
\item[{\rm{(ii)}}]
$\dim_{\mathbb{C}}D \isawatyp =1$.  
\item[{\rm{(iii)}}]
The parameter $(i,j,\lambda,\nu,\delta,\varepsilon)$ satisfies 
\begin{align}
\label{eqn:ijcond}
&\{j,n-j-1\} \cap \{i-2, i-2,i,i+1\} \ne \emptyset,
\\
&\nu-\lambda \in {\mathbb{N}}, 
\notag
\\
 &\text{a certain condition $Q\equiv Q_{i,j}$
 on $(\lambda,\nu,\delta,\varepsilon)$.  }
\label{eqn:addQ}
\end{align}
\end{enumerate}
\end{theorem}
The first condition \eqref{eqn:ijcond} concerns the degrees
 $i$ and $j$
 of differential forms.  
Loosely speaking,
 conformally covariant differential SBOs exist
 only if the degrees $i$ and $j$ are close 
 to each other or the sum $i+j$ is close to $n$.  
The last \lq\lq{additional}\rq\rq\ condition $Q_{i,j}$ depends on $(i,j)$.  
We give the condition $Q_{i,j}$ explicitly 
 in the following two cases:
\begin{enumerate}
\item[$\bullet$] Case $j=i$.  
$Q_{i,i}$ amounts to $\nu \in {\mathbb{C}}$ and $\delta \equiv \varepsilon \equiv \nu-\lambda \mod 2$.  
\item[$\bullet$] Case $j=i+1$.  
For $1 \le i \le n-2$, 
$Q_{i,i+1}$ amounts to $(\lambda,\nu)=(0,0)$
 and $\delta \equiv \varepsilon \equiv 0 \mod 2$;
for $i=0$, 
 $Q_{0,1}$ amounts to $\lambda \in -{\mathbb{N}}$, 
 $\nu=0$, and $\delta \equiv \varepsilon \equiv \lambda \mod 2$.  
\end{enumerate}
See \cite[Thm.~1.1]{KKP} for the precise conditions
 in the other remaining six cases.

Second,
 we go on with Problem \ref{prob:conf} (3) (Stage C)
 about the construction
 of symmetry breaking operators.  
For this we work with the pair $({\mathbb{R}}^n, {\mathbb{R}}^{n-1})$
 of the flat Riemannian manifolds 
 which are conformal to $(S^n \setminus \{\operatorname{pt}\}, S^{n-1} \setminus \{\operatorname{pt}\})$
 via the stereographic projection.

We begin with a {\it{scalar-valued}} operator
 (Juhl's operator, \cite{Juhl}).  
Suppose that our hyperplane $Y={\mathbb{R}}^{n-1}$ of $X={\mathbb{R}}^{n}$
 is defined by $x_n=0$ in the coordinates $(x_1, \cdots, x_n)$.  
For $\mu \in {\mathbb{C}}$ and $k \in {\mathbb{N}}$, 
 we define a homogeneous differential operator
 of order $k$ by 
\[
  {\mathcal{D}}_k^{\mu} :=
  \sum_{0 \le i \le [\frac k 2]} a_i(\mu)
  (-\Delta_{{\mathbb{R}}^{n-1}})^i
  \frac{\partial^{k-2i}}{\partial x_n^{k-2i}}
  \colon C^{\infty}({\mathbb{R}}^n) \to C^{\infty}({\mathbb{R}}^n), 
\]
 where $\{a_i(\mu)\}$ are the coefficients
 of the Gegenbauer polynomial:
\[
   C_k^{\mu}(t) = \sum_{0 \le i \le [\frac k 2]} a_i(\mu) t^{k-2i}.  
\]
Building on the scalar-valued operators,
 we introduced in \cite{KKP}
 {\it{matrix-valued}} differential symmetry breaking operators
\[
   {\mathcal{D}}_{\lambda,k}^{i \to j}
   \colon
  {\mathcal{E}}^i({\mathbb{R}}^{n}) \to {\mathcal{E}}^j({\mathbb{R}}^{n-1})
\]
for each pair $(i,j)$ satisfying \eqref{eqn:ijcond}.  
We illustrate a concrete formula when $j=i$. 
We set 
\[
   {\mathcal{D}}_{\lambda,k}^{i \to i}
   :=
   {\operatorname{Rest}}_{x_n=0}
   \circ
   ({\mathcal{D}}_{k-2}^{\mu+1} d d^{\ast}
    + a {\mathcal{D}}_{k-1}^{\mu} d \iota_{\frac{\partial}{\partial x_n}}
   + b {\mathcal{D}}_k^{\mu}), 
\]
 where $d^{\ast}$ is the codifferential, 
$
     \iota_{\frac{\partial}{\partial x_n}} 
     \colon 
    {\mathcal{E}}^i({\mathbb{R}}^{n}) 
    \to 
    {\mathcal{E}}^j({\mathbb{R}}^{n-1})
$
 is the inner multiplication
 of the vector field $\frac{\partial}{\partial x_n}$,
 and 
\[
  a:=
  \begin{cases}
  1 
&\text{($k$: odd)}
\\
\lambda+i-\frac n 2+k
&\text{($k$: even)}
  \end{cases}, 
\quad
b:=\frac{\lambda+k}{2},
\quad
\mu:=\lambda+i-\frac{n-1}{2}.  
\]
Thus the operator ${\mathcal{D}}_{\lambda,k}^{i \to i}$ is obtained
 as the composition of a $\invHom {\mathbb{C}}{\Exterior^i({\mathbb{C}}^n)}{\Exterior^i({\mathbb{C}}^n)}$-valued homogeneous 
 differential operator on ${\mathbb{R}}^n$
 of order $k$
 with the restriction map to the hyperplane ${\mathbb{R}}^{n-1}$.

The matrix-valued differential operators
 ${\mathcal{D}}_{\lambda,k}^{i \to j} \colon   {\mathcal{E}}^i({\mathbb{R}}^{n}) \to {\mathcal{E}}^j({\mathbb{R}}^{n-1})$
 were defined in \cite[Chap.~1]{KKP} also 
 for the other seven cases 
 when the condition (iii) in Theorem \ref{thm:Dnonzero} is fulfilled.  

\vskip 0.8pc
{\bf{Methods of proof}}
 in finding the formul{\ae}
 for ${\mathcal{D}}_{\lambda,k}^{i \to j}$.  
The approach in \cite{KKP} is based 
 on the {\it{F-method}} \cite{xkHelg85}, 
 which reduces a problem of finding the operators
 ${\mathcal{D}}_{\lambda,k}^{i \to j}$
 to another problem of finding polynomial solutions
 to a system of ordinary differential equations
 ({\it{F-system}}).  
An alternative approach for $j=i-1,i$ is given
 in \cite{xkresidue}
 by taking the residues of the regular symmetry breaking operators
 (see also Section \ref{sec:CSBO} below).

\vskip 0.8pc
With the aforementioned operators
 ${\mathcal{D}}_{\lambda,k}^{i \to j}$, 
 Problem \ref{prob:conf} (3)
 for differential operators were solved in \cite[Thms.~1.4--1.8]{KKP}, 
 which may be thought of as an answer to Stage C
 of branching problems.  
We illustrate the results
 with the following two theorems
 in the case
 where $j=i$ and $i+1$.  
\begin{theorem}
[$j=i$ case]
\label{thm:SBOji}
Suppose $\nu \in {\mathbb{C}}$, $k:=\nu-\lambda \in {\mathbb{N}}$, 
 and $\delta \equiv \varepsilon \equiv k \mod 2$.  
\begin{enumerate}
\item[{\rm{(1)}}]
The linear map ${\mathcal{D}}_{\lambda,k}^{i \to i}$ extends to a conformally covariant symmetry breaking operator from 
$
   {\mathcal{E}}^i(S^{n})_{\lambda,\delta}
$
to 
$
   {\mathcal{E}}^i(S^{n-1})_{\nu,\varepsilon}.
$

\item[{\rm{(2)}}]
Conversely,
 any conformally covariant differential symmetry breaking operator from 
 ${\mathcal{E}}^i(S^{n})_{\lambda,\delta}$ to 
 ${\mathcal{E}}^i(S^{n-1})_{\nu,\varepsilon}$
 is proportional to ${\mathcal{D}}_{\lambda,k}^{i \to i}$, 
 or its renormalization (\cite[(1.10)]{KKP}).  
\end{enumerate}
\end{theorem}
\begin{theorem}
[$j=i+1$ case]
\label{thm:SBOji1}
\begin{enumerate}
\item[{\rm{(1)}}]
Suppose $1 \le i \le n-2$, 
 $(\lambda,\nu)=(n-2i,n-2i+3)$, 
 and $\delta \equiv \varepsilon \equiv 1 \mod 2$.  
Then the linear map
\[
     {\operatorname{Rest}}\circ d \colon 
 {\mathcal{E}}^i(S^{n})_{\lambda,\delta} \to 
 {\mathcal{E}}^{i+1}(S^{n-1})_{\nu,\varepsilon}
\]
 is a conformally covariant SBO.  
Conversely, 
 a nonzero conformally covariant differential SBO from 
 ${\mathcal{E}}^i(S^{n})_{\lambda,\delta}$ to 
 ${\mathcal{E}}^{i+1}(S^{n-1})_{\nu,\varepsilon}$
 exists only for the above parameters,
 and such an operator 
 is proportional to ${\operatorname{Rest}}\circ d$.  
\item[{\rm{(2)}}]
Suppose $i=0$, $\lambda \in \{0,-1,-2,\cdots\}$,
 $\nu=0$, 
 and $\delta \equiv \varepsilon \equiv \lambda \mod 2$.  
Then the linear map
\[
   {\operatorname{Rest}}_{x_n=0}\circ {\mathcal{D}}_{-\lambda}^{\lambda-\frac{n-1}{2}} \circ d \colon 
 {\mathcal{E}}^0({\mathbb{R}}^{n}) \to 
 {\mathcal{E}}^{1}({\mathbb{R}}^{n-1})
\]
 extends to a conformally covariant SBO from ${\mathcal{E}}^0(S^{n})_{\lambda,\delta}$ to ${\mathcal{E}}^{1}(S^{n-1})_{0,\varepsilon}$.  
Conversely, 
 a nonzero conformally covariant differential SBO from ${\mathcal{E}}^0(S^{n})_{\lambda,\delta}$ to ${\mathcal{E}}^{1}(S^{n-1})_{\nu,\varepsilon}$
 exists only for the above parameters,
 and such an operator is proportional to the above operator.  
\end{enumerate}
\end{theorem}

\begin{remark}
\begin{enumerate}
\item[{\rm{(1)}}]
By using the Hodge $\ast$ operator on $X$ or its submanifold $Y$, 
 the other six cases can be reduced to either the $j=i$ case
 (Theorem \ref{thm:SBOji}) or the $j=i+1$ case (Theorem \ref{thm:SBOji1}).  
The construction and classification of differential symmetry breaking operators
 in the model space \eqref{eqn:XYsph} is thus completed.  
Its generalization to the pseudo-Riemannian case is proved in \cite{KKP2}.  
\item[{\rm{(2)}}]
Special cases of Theorem \ref{thm:SBOji} were known earlier.  
The case $j=i=0$ (scalar-valued case) was discovered
 by A.~Juhl \cite{Juhl}.  
Different approaches have been proposed
 by Fefferman--Graham \cite{FG13}, 
 Kobayashi--{\O}rsted--Sou{\v c}ek--Somberg \cite{KOSS},
 and Clerc \cite{Cl} among others.  
Our approach uses an algebraic Fourier transform
 of Verma modules
 ({\it{F-method}}), 
 see \cite{xkHelg85, KP1}.  
\item[{\rm{(3)}}]
The case $n=2$ is closely
 related to the celebrated Rankin--Cohen bidifferential operator
 via holomorphic continuation \cite{KP2}.  
\end{enumerate}
\end{remark}

\section{Classification theory: nonlocal conformally covariant SBOs}
\label{sec:CSBO}
In this section 
 we consider nonlocal operators such as integral operators as well,
and thus complete the classification problem 
 (Problem \ref{prob:conf})
 for the model space $(X,Y)=(S^n,S^{n-1})$.

Building on the classification results
 on $D \isawatyp$ in Section \ref{sec:DSBO}, 
 we want to 
\begin{enumerate}
\item[$\bullet$]
find $\dim_{\mathbb{C}} H\isawatyp/D \isawatyp$;
\item[$\bullet$]
find a basis in $H\isawatyp$ modulo $D \isawatyp$. 
\end{enumerate}
This idea fits well with the general strategy
 to understand the whole space of symmetry breaking operators
 between principal series representations
 of a reductive group 
 and its subgroup $G'$
 by using the filtration 
 given by the support of distribution kernels
 \cite[Chap.~11, Sec.~2]{sbon}.  
Thus we start with the general setting
 where $(G,G')$ is a pair of real reductive Lie groups.  
Let $P=M A N$ and $P'=M' A' N'$ be Langlands decompositions of minimal parabolic subgroups 
 of $G$ and $G'$, 
respectively.  
For an irreducible representation $(\sigma,V)$ of $M$
 and a one-dimensional representation ${\mathbb{C}}_{\lambda}$ of $A$, 
 we define a principal series representation of $G$
 by unnormalized parabolic induction
\[
  I(\sigma,\lambda):={\operatorname{Ind}}_P^G(\sigma \otimes {\mathbb{C}}_{\lambda} \otimes {\bf{1}}).  
\]
Similarly,
 we define that of the subgroup $G'$, 
 to be denoted by 
\[
  J(\tau,\nu):={\operatorname{Ind}}_{P'}^{G'}(\tau \otimes {\mathbb{C}}_{\nu} \otimes {\bf{1}})
\]
 for an irreducible representation $(\tau,W)$ of $M'$
 and a one-dimensional representation ${\mathbb{C}}_{\nu}$ of $A'$.

By abuse of notation,
 we identify a representation
 with its representations space,
 and set $V_{\lambda}:=V \otimes {\mathbb{C}}_{\lambda}$
 and $W_{\nu}:=W \otimes {\mathbb{C}}_{\nu}$.  
Let ${\mathcal{V}}_{\lambda}^{\ast}$ be the dualizing bundle
 of the $G$-homogeneous bundle $G \times_P V_{\lambda}$
 over the real flag manifold $G/P$.  
Then there is a natural linear bijection 
 between the space of symmetry breaking operators
 and the space of invariant distributions
 (see \cite[Prop.~3.2]{sbon}):
\begin{equation}
\label{eqn:SBO}
{\operatorname{Hom}}_{G'}(I_{\delta}(\sigma,\lambda)|_{G'}, J_{\varepsilon}(\tau,\nu))
\overset \sim \rightarrow
({\mathcal{D}}'(G/P, {\mathcal{V}}_{\lambda}^{\ast}) \otimes W_{\nu})^{\Delta(P')}, 
\quad
T \mapsto K_T,
\end{equation}

Suppose now that the condition (PP) in Theorem \ref{thm:PPBB} is fulfilled.  
Then this implies
 that $\#(P' \backslash G/P)<\infty$, 
see \cite[Rem.~2.5 (4)]{xKOfm}.  
We denote by $\{Z_{\alpha}\}$ the totality
 of $P'$-orbits on $G/P$.  
We define a partial order $\alpha \prec \beta$
 by $Z_{\alpha} \subset \overline {Z_{\beta}}$, 
 the closure of $Z_{\beta}$ in $G/P$.  
Then there is the unique minimal index $\alpha_{\operatorname{min}}$
 corresponding to the closed $P'$-orbit in $G/P$, 
 and maximal ones $\beta_1$, $\cdots$, $\beta_N$
 corresponding to open $P'$-orbits in $G/P$.

We observe
 that the support ${\operatorname{Supp}}(K_T)$
 of the distribution kernel $K_T$ is a closed $P'$-invariant subset
 of $G/P$, 
 and accordingly,
 define 
\[
  H(\alpha)\equiv H_{\tau,\nu}^{\sigma,\lambda}(\alpha)
:=
  \{T \in {\operatorname{Hom}}_{G'}(I(\sigma,\lambda)|_{G'}, J(\tau,\nu)):
  {\operatorname{Supp}}(K_T) \subset \overline {Z_{\alpha}}
\}
\]
via the isomorphism \eqref{eqn:SBO}.  
Clearly, 
$ 
H(\alpha)  \subset H(\beta)
$
 if $\alpha \prec \beta$.  
It follows from \cite[Lem.~2.3]{KP1} that 
\[
H(\alpha_{\operatorname{min}}) ={\operatorname{Diff}}_{G'}(I(\sigma,\lambda)|_{G'}, J(\tau,\nu)).  
\]
In contrast to the smallest support $Z_{\alpha_{\operatorname{min}}}$,
 a symmetry breaking operator $T$ is called {\it{regular}}
 (\cite[Def.~3.3]{sbon})
 if ${\operatorname{Supp}}(K_T)$ contains $Z_{\beta_j}$
 for some $1 \le j \le N$.

We now return to the special setting \eqref{eqn:Lorentz}.  
Then the Levi subgroup $MA$ of the minimal parabolic subgroup $P=MAN$
 of $G=O(n+1,1)$ is given by $(O(n) \times O(1)) \times {\mathbb{R}}$.  
For $0 \le i \le n$, 
 $\delta \in \{\pm\}$, 
 and $\lambda \in {\mathbb{C}}$, 
 we consider the outer tensor product representation
 $\Exterior^i({\mathbb{C}}^n) \boxtimes \delta \boxtimes {\mathbb{C}}_{\lambda}$
 of $M A$, 
 and extend it to $P$
 by letting $N$ act trivially.  
The resulting $P$-module is denoted simply by $\Exterior^i({\mathbb{C}}^n) \otimes \delta \otimes {\mathbb{C}}_{\lambda}$.  
We define an unnormalized principal series representation 
 of $G=O(n+1,1)$ by
\[
  I_{\delta}(i,\lambda)
  \equiv I(\Exterior^i({\mathbb{C}}^n) \boxtimes \delta, {\lambda})
:={\operatorname{Ind}}_P^G
 (\Exterior^i({\mathbb{C}}^n) \otimes \delta \otimes {\mathbb{C}}_{\lambda}).  
\]
\begin{lemma}
\label{lem:ps}
Let $0 \le i \le n$, $\delta \in \{\pm\}$, $\lambda \in {\mathbb{C}}$.  
\begin{enumerate}
\item[{\rm{(1)}}]
The $G$-module $I_{\delta}(i,\lambda)$ is irreducible
 if $\lambda \not \in {\mathbb{Z}}$.  
\item[{\rm{(2)}}]
There is a natural isomorphism
 ${\mathcal{E}}^i(S^{n})_{\lambda,\delta}
 \simeq I_{(-1)^i \delta}(i,\lambda+i)$
 as $G$-modules.  
\end{enumerate}
\end{lemma}

For the proof of Lemma \ref{lem:ps}~(2), 
 see {\cite[Prop.~2.3]{KKP}}.  

Lemma \ref{lem:ps} (2) suggests
 that we can reformulate Problem \ref{prob:conf}
 about differential forms
 on the pair of conformal manifolds \eqref{eqn:XYsph}
 into a question of symmetry breaking operators
 between principal series representations
 for the pair \eqref{eqn:Lorentz} of reductive groups.  
We write $\widetilde D$ and $\widetilde H$
 if we use $I_{\delta}(i,\lambda)$ and $J_{\varepsilon}(j,\nu)
={\operatorname{Ind}}_{P'}^{G'}(\Exterior^j({\mathbb{C}}^{n-1})\otimes \varepsilon \otimes {\mathbb{C}}_{\nu})$
 instead of $D$ and $H$
 in \eqref{eqn:D} and \eqref{eqn:H}, 
 respectively.  
By Lemma \ref{lem:ps} (2), 
 we have 
\[
  H \isawatyp =\widetilde H \isawavar i {\lambda+i} {(-1)^i \delta} j {\nu+j}{(-1)^j \varepsilon}
\]
and similarly for $D$ and $\widetilde D$.  
Thus we want to 
\begin{enumerate}
\item[$\bullet$]
find $\dim_{\mathbb{C}} \widetilde H \isawatyp/\widetilde D \isawatyp$;
\item[$\bullet$]
find a basis in $\widetilde H \isawatyp$ modulo $\widetilde D \isawatyp$. 
\end{enumerate}
First, we obtain:
\begin{theorem}[localness theorem]
\label{thm:4A}
If $j \ne i-1$ or $i$, 
 then 
\[
   \widetilde H \isawatyp = \widetilde D \isawatyp.
\]
\end{theorem}
In the setting \eqref{eqn:Lorentz}, 
there exists a unique open $P'$-orbit in $G/P$, 
 and accordingly,
 there exists at most one family
 of (generically) regular symmetry breaking operators from the $G$-modules
 $I_{\delta}(i,\lambda)$ to the $G'$-modules $J_{\varepsilon}(j,\nu)$.  
We prove
 that such a family exists
 if and only if $j=i-1$ or $i$, 
 and it plays a crucial role 
 in the classification problem
 of SBOs modulo the space
 $\widetilde D \isawatyp$
 of differential SBOs as follows.  
We introduce the set of \lq\lq{special parameters}\rq\rq\ by 
\begin{multline}
\label{eqn:singset}
  \Psi_{\operatorname{sp}}
  :=
  \left\{(\lambda,\nu, \delta, \varepsilon) \in {\mathbb{C}}^2 \times \{\pm\}^2
   :
   \text{$\nu-\lambda \in 2 {\mathbb{N}}$ when $\delta \varepsilon =1$}
  \right.
\\
  \left.
   \text{ or 
         $\nu-\lambda \in 2 {\mathbb{N}}+1$ when $\delta \varepsilon =-1$}
   \right\}.  
\end{multline}

\begin{theorem}
\label{thm:4B}
Suppose $j=i-1$ or $i$, 
 and $\delta, \varepsilon \in \{\pm\}$.  
Then there exists a family of continuous $G'$-homomorphism
\[
  \Atbb \lambda \nu {\delta\varepsilon}{i,j}
  \colon 
  I_{\delta}(i,\lambda) \to J_{\varepsilon}(j,\nu)
\]
such that $\Atbb \lambda \nu {\delta\varepsilon}{i,j}$ depends holomorphically
 on $(\lambda,\nu)\in {\mathbb{C}}^2$
 and that the set of the zeros of $\Atbb \lambda \nu {\delta\varepsilon}{i,j}$
 is discrete in $(\lambda,\nu)\in {\mathbb{C}}^2$.  

\begin{enumerate}
\item[{\rm{(1)}}]
If $(\lambda,\nu, \delta, \varepsilon) \not \in \Psi_{\operatorname{sp}}$
 then $\Atbb \lambda \nu {\delta\varepsilon}{i,j} \ne 0$
 and 
\[
 \widetilde H \isawatyp
 = {\mathbb{C}} \Atbb \lambda \nu {\delta\varepsilon}{i,j} \supsetneqq
   \widetilde D \isawatyp=\{0\}.  
\]
\item[{\rm{(2)}}]
If $(\lambda,\nu, \delta, \varepsilon) \in \Psi_{\operatorname{sp}}$
 and $\Atbb \lambda \nu {\delta\varepsilon}{i,j} \ne 0$, 
 then 
\[
 \widetilde H \isawatyp
 = \widetilde D \isawatyp.  
\]
\item[{\rm{(3)}}]
If $(\lambda,\nu, \delta, \varepsilon) \in \Psi_{\operatorname{sp}}$
 and $\Atbb \lambda \nu {\delta\varepsilon}{i,j} = 0$, 
 then 
\[
 \dim_{\mathbb{C}} \widetilde H \isawatyp
 = \dim \widetilde D \isawatyp +1.  
\]
\end{enumerate}
\end{theorem}
The discrete set
$
 \{
  (i,j,\lambda,\nu,\delta,\varepsilon):
  \Atbb \lambda \nu {\delta\varepsilon}{i,j} =0
 \}
$
 has been determined in \cite{xkresidue}, 
 and thus the classification of conformally covariant symmetry breaking operators
\[
  {\mathcal{E}}^i(X)_{\lambda,\delta} \to {\mathcal{E}}^j(Y)_{\nu,\varepsilon}
\]
for the model space $(X,Y)=(S^n,S^{n-1})$ 
 is accomplished.  
A detailed proof for the classification 
 together with some important properties
 of symmetry breaking operators
 (Stage C)
 such as
\begin{enumerate}
\item[$\bullet$]
$(K,K')$-spectrum
 (a generalized eigenvalue), 
\item[$\bullet$]
functional equations, 
\item[$\bullet$]
residue formul{\ae}, 
\end{enumerate}
 will be given in separate papers
 (see \cite{xkresidue} for the residue formul{\ae}, 
 and \cite{sbonvec} for the classification).

\section{Application to periods and automorphic form theory}
\label{sec:period}
Let $G$ be a reductive group,
 and $H$ a reductive subgroup.  
\begin{definition}
\label{def:7.1}
An irreducible admissible smooth representation $\Pi$ of $G$
 is {\it{$H$-distinguished}}
 if $\invHom H {\Pi|_{H}}{\mathbb{C}} \ne \{0\}$.  
In this case,
 it is also said
 that $\Pi$ has an {\it{$H$-period}}.  
By the Frobenius reciprocity theorem, 
 the condition is equivalent
 to $\invHom G \Pi {C^{\infty}(G/H)} \ne \{0\}$.  
\end{definition}
In this section,
 we discuss an application
 of symmetry breaking operators
 to find periods (Definition \ref{def:7.1})
 of irreducible unitary representations.  
We highlight the case
 when $\Pi$ has nonzero $({\mathfrak{g}},K)$-cohomologies.  
The motivation comes from automorphic form theory, 
 of which we now recall a prototype.

\begin{fact}
[Matsushima--Murakami, {\cite{BW}}]
Let $\Gamma$ be a cocompact discrete subgroup of $G$.  
Then we have
\[
   H^{\ast}(\Gamma \backslash G/K;{\mathbb{C}})
   \simeq
   \bigoplus_{\Pi \in \widehat G}
   {\operatorname{Hom}}_G
   (\Pi,L^2(\Gamma\backslash G)) \otimes H^{\ast}({\mathfrak {g}},K;\Pi_K).  
\]
\end{fact}
The left-hand side gives topological invariants
 of the locally symmetric space
 $M=\Gamma\backslash G/K$, 
 whereas the right-hand side is described in terms of the representation theory.  
We note that ${\operatorname{Hom}}_G
   (\Pi,L^2(\Gamma\backslash G))$ is finite-dimensional 
 for all $\Pi \in \widehat G$
 by a theorem of Gelfand--Piateski-Shapiro,
 and the sum is taken over the following finite set 
\[
   \widehat G_{\operatorname{cohom}}
   :=
   \{\Pi \in\widehat G
   : H^{\ast}({\mathfrak {g}},K;\Pi_K) \ne \{0\}
   \}, 
\]
 which was classified by Vogan and Zuckerman \cite{VZ}.

In the case where $G=O(n+1,1)$, 
 there are $2(n+1)$ elements
 in $\widehat G_{\operatorname{cohom}}$.  
Following the notation in \cite[Thm.~2.6]{KKP}, 
 we label them as 
\[
  \{\Pi_{\ell,\delta} : 0 \le \ell \le n+1, \delta \in \{\pm\}\}, 
\]
and we define
\begin{alignat*}{2}
  {\operatorname{Index}}\equiv {\operatorname{Index}}_G&\colon \dual G {cohom} \to \{0,1,\cdots,n+1\},
\quad
&&\Pi_{\ell, \delta} \mapsto \ell,
\\
{\operatorname{sgn}}\equiv {\operatorname{sgn}}_G &\colon \dual G {cohom} \to \{\pm\},
\quad
&&\Pi_{\ell, \delta} \mapsto \delta.  
\end{alignat*}
We illustrate the labeling
 by two examples:
\begin{example}
[one-dimensional representations]
\label{ex:chi}
There are four one-dimensional representations of $G$, 
 which are given as
\[
 \{\Pi_{0,+}\simeq {\bf{1}}, \Pi_{0,-}, \Pi_{n+1,+}, \Pi_{n+1,-} \simeq \det\}.  
\]
\end{example}
\begin{example}
[tempered representations]
\label{ex:tempered}
For $n$ odd
 $\Pi$ is the smooth representation
 of a discrete series representation of $G$
 iff ${\operatorname{Index}}(\Pi)=\frac 1 2 (n+1)$,  
 whereas for $n$ even $\Pi$ is  that of tempered representation of $G$
 iff ${\operatorname{Index}}(\Pi) \in \{\frac n 2, \frac n 2+1\}$.  
\end{example}

We give a necessary and sufficient condition
 for the existence of symmetry breaking operators
 between irreducible representations
 of $G$ and those of the subgroup $G'$
 with nonzero $({\mathfrak{g}},K)$-cohomologies:
\begin{theorem}
[{\cite{sbonvec}}]
Let $(G,G')=(O(n+1,1),O(n,1))$,
 and $(\Pi, \pi)
 \in \widehat G_{\operatorname{cohom}} \times \widehat G_{\operatorname{cohom}}'$.  
Then the following three conditions on $(\Pi, \pi)$
 are equivalent.  
\begin{enumerate}
\item[{\rm{(i)}}]
${\operatorname{Hom}}_{G'}
   (\Pi^{\infty}|_{G'},\pi^{\infty}) \ne \{0\}$.  
\item[{\rm{(ii)}}]
The outer tensor product representation
 $\Pi^{\infty} \boxtimes \pi^{\infty}$ is ${\operatorname{diag}}(G')$-distinguished.  
\item[{\rm{(iii)}}]
${\operatorname{Index}}_{G}(\Pi)-1
 \le 
 {\operatorname{Index}}_{G'}(\pi)
 \le 
 {\operatorname{Index}}_{G}(\Pi)
$
 and ${\operatorname{sgn}}(\Pi)={\operatorname{sgn}}(\pi)$.  
\end{enumerate}
\end{theorem}

The proof uses the symmetry breaking operators 
 that are discussed in Section \ref{sec:CSBO}
 and the relationship between $\dual G {cohom}$
 and conformal representations on differential forms
 on the sphere $S^n$ summarized as below.  
\begin{lemma}
[{\cite[Thm.~2.6]{KKP}}]
If $\Pi \in \widehat G_{\operatorname{cohom}}$, 
 then $\Pi^{\infty}$ can be realized 
 as a subrepresentation
 of ${\mathcal{E}}^i(S^n)_{0,\delta}$
 with $i={\operatorname{Index}}_{G}(\Pi)$
 and $\delta =(-1)^i{\operatorname{sgn}}_G(\Pi)$
 if ${\operatorname{Index}}_{G}(\Pi) \ne n+1$, 
 and also as a quotient of ${\mathcal{E}}^i(S^n)_{0,\delta}$
 with $i={\operatorname{Index}}_{G}(\Pi)-1$
 and $\delta =(-1)^i{\operatorname{sgn}}_G(\Pi)$
 if ${\operatorname{Index}}_{G}(\Pi) \ne 0$.  
\end{lemma}

To end this section,
 we consider a tower of subgroups of a reductive group $G$:
\[
\{e\}= G^{(0)} \subset G^{(1)} \subset \cdots \subset G^{(n)} \subset G^{(n+1)}=G.  
\]
Accordingly,
 there is a family of homogeneous spaces with $G$-equivariant quotient maps:
\[
   G=G/G^{(0)} \to G/G^{(1)} \to \cdots \to G/G^{(n+1)} = \{{\operatorname{pt}}\}.  
\]
In turn,
 we have natural inclusions of $G$-modules:
\[
  C^{\infty}(G) = C^{\infty}(G/G^{(0)}) \supset C^{\infty}(G/G^{(1)}) \supset 
  \cdots \supset C^{\infty}(G/G^{(n+1)}) ={\mathbb{C}}.  
\]
A general question is:
\begin{problem}
Let $\Pi \in \widehat G_{\operatorname{smooth}}$.  
Find $k$ as large as possible 
 such that $\Pi$ is $G^{(k)}$-distinguished, 
 or equivalently,
such that the smooth representation $\Pi^{\infty}$ can be realized in $C^{\infty}(G/G^{(k)})$.  
\end{problem}

Any irreducible admissible smooth representation of $G$
 can be realized in the regular representation
 on $C^{\infty}(G/G^{(0)}) \simeq C^{\infty}(G)$
 via matrix coefficients,
 whereas irreducible representations
 that can be realized in $C^{\infty}(G/G^{(0)})={\mathbb{C}}$
 is the trivial one-dimensional representation ${\bf{1}}$.

Suppose that $G=O(n+1,1)$, 
 and consider a chain of subgroups of $G$ by
\[
   G^{(k)}:=O(k,1)
\quad
   (0 \le k \le n+1).  
\]
Then $G^{(n+1)}=G$, 
 however, 
 $G^{(0)}$ is not exactly $\{e\}$ 
 but $G^{(0)} =O(1)$ is a finite group of order two.  
Accordingly,
 we consider $\Pi \in \dual G {cohom}$
 with ${\operatorname{sgn}}(\Pi)=+$ below.  
\begin{theorem}
\label{thm:maxperiod}
Suppose $\Pi \in \widehat G_{\operatorname{cohom}}$
 with ${\operatorname{sgn}}(\Pi)=+$.    
Then 
\[
 {\operatorname{Hom}}_G(\Pi^{\infty}, C^{\infty}(G/G^{(k)})) \ne \{0\}
\quad
\text{for all $k \le n+1-{\operatorname{Index}}_G(\Pi)$.}
\]
\end{theorem}

\begin{example}
[one-dimensional representations]
\label{ex:tiny}
Suppose $\Pi \in \dual G {cohom}$
 with ${\operatorname{sgn}}_G(\Pi)=+$.  
We consider two opposite extremal cases,
 {\it{i.e.}}, 
 ${\operatorname{Index}}_G(\Pi)=0$
 and $=n+1$.  
If ${\operatorname{Index}}_G(\Pi)=0$, 
 then $\Pi$ is isomorphic to the trivial one-dimensional representation ${\bf{1}}$, 
 and can be realized in $C^{\infty}(G/G^{(k)})$
 for all $0 \le k \le n+1$
 as in Theorem \ref{thm:maxperiod}.  
On the other hand, 
 if ${\operatorname{Index}}_G(\Pi)=n+1$, 
 then $\Pi$ is another one-dimensional representation of $G$
 ($\Pi_{n+1,+} \simeq \chi_{-+}$ with the notation \cite[(2.9)]{KKP}).  
In this case, 
 $\Pi$ can be realized in $C^{\infty}(G/G^{(k)})$
 iff $k=0$, 
namely,
 iff $G^{(k)}=O(1)$.  
\end{example}

\begin{remark}
The size of an (infinite-dimensional) representation could be measured by
its Gelfand--Kirillov dimension, or more precisely, by its associated variety
or by the partial flag variety for which its localization can be realized
as a $\mathcal D$-module.
Then one might expect the following assertion: 
\begin{multline}
\label{eqn:smallrep}
\text{{\sl{the larger the isotropy subgroup 
 $G^{(k)}$ is 
 ({\it{i.e.,}} the larger $k$ is),}}}
\\
\text{{\sl{the \lq\lq{smaller}\rq\rq\ irreducible subrepresentations
 of $C^{\infty}(G/G^{(k)})$ become.
}}}
\end{multline}

This is reflected partially in Theorem \ref{thm:maxperiod}, 
 however, 
 Theorem \ref{thm:maxperiod} asserts even sharper results.  
To see this, 
 we set
\[
 r:= \min({\operatorname{Index}}_G(\Pi), n+1-{\operatorname{Index}}_G(\Pi)).  
\]
Then the underlying $({\mathfrak{g}},K)$-module $\Pi_K$
 can be expressed as a cohomological parabolic induction from a $\theta$-stable parabolic subalgebra ${\mathfrak {q}}_r$
 with Levi subgroup $N_G({\mathfrak {q}}_r) \simeq SO(2)^r \times O(n+1-2r,1)$
 (\cite{KV}, see also \cite[Thm.~3]{KMemoirs92}).  
Theorem \ref{thm:maxperiod} tells
 that if $n+1 \le 2k$, 
 then the larger $k$ is, 
 the smaller $r = {\operatorname{Index}}_G(\Pi)$ becomes,
 namely,
 the smaller the $({\mathfrak{g}},K)$-modules
 that are cohomologically parabolic induced modules from ${\mathfrak {q}}_r$ become.  
This matches \eqref{eqn:smallrep}.  
On the other hand, 
 if $2k \le n+1$, 
 then the constraints in Theorem~\ref{thm:maxperiod}
 provide an interesting phenomenon which is opposite to \eqref{eqn:smallrep}
 because $r=n+1-{\operatorname{Index}}_G(\Pi)$, 
 and thus suggest sharper estimates
 than \eqref{eqn:smallrep}.  
For instance,
 the representation $\Pi_{n+1,+} (\simeq \chi_{-+})$
 is \lq\lq{small}\rq\rq\
 because it is one-dimensional, 
 but it can be realized in $C^{\infty}(G/G^{(k)})$
 only for $k=0$ 
 as we saw in Example \ref{ex:tiny}.  
\end{remark}

\begin{remark}[comparison with $L^2$-theory]
\label{rem:temp}
Theorem \ref{thm:maxperiod} implies
 that the smooth representation $\Pi^{\infty}$
 of a tempered representation $\Pi$
with nonzero $(\mathfrak g, K)$-cohomologies
 (see Example \ref{ex:tempered})
 occurs in $C^{\infty}(G/G^{(k)})$
 if $k \le \frac n 2+1$.  
On the other hand,
 for a reductive homogeneous space $G/H$, 
 a general criterion for the unitary representation $L^2(G/H)$
 to be tempered was proved in a joint work \cite{BK} with Y.~Benoist
 by a geometric method.  
In particular,
the unitary representation $L^2(G/G^{(k)})$ is tempered
 if and only if $k \le \frac n 2 +1$, 
see \cite[Ex.~5.10]{BK}.  
\end{remark}

\vskip 2pc
{\bf{Acknowledgements}}:\enspace
The author was partially supported by Grant-in-Aid for Scientific
Research (A) (25247006), Japan Society for the Promotion of Science.
This article is based on the plenary lecture
 that the author delivered at the XXXVI-th workshop
 on Geometric Methods in Physics at Bialowieza in 2017.  
He expresses his gratitude to the organizers
 for their warm hospitality.

\end{document}